\documentclass[12pt,twoside]{amsart}
\usepackage[dvipsnames]{xcolor}
\usepackage{fancyhdr}
\usepackage{amsmath,amsfonts,amsbsy,amsgen,amscd,mathrsfs,amssymb,amsthm}
\usepackage{pifont}
\usepackage{mathtools}

\usepackage{xcolor} %
\definecolor{cit}{rgb}{0.91,0.39,0.16}	%
\definecolor{dark-gray}{gray}{0.3}
\definecolor{dkgray}{rgb}{.3,.3,.3}
\definecolor{medgray}{rgb}{.5,.5,.5}
\definecolor{ltgray}{rgb}{.7,.7,.7}
\definecolor{dkblue}{rgb}{0,0,.5}
\definecolor{medblue}{rgb}{0,0,.75}
\definecolor{ltblue}{rgb}{0.97,0.97,1}
\definecolor{rust}{rgb}{0.5,0.1,0.1}
\definecolor{ltyellow}{rgb}{1, 1, 0.9}

\usepackage[style=alphabetic, citestyle=alphabetic,
giveninits=true,
sorting=nyt, sortcites=false,
autopunct=true, autolang=hyphen,
hyperref=true,
abbreviate=false,
backref=false,
backend=bibtex
]{biblatex}
\defbibheading{bibempty}{}

\renewbibmacro{in:}{%
  \ifentrytype{article}{}{\printtext{\bibstring{in}\intitlepunct}}}

\AtEveryBibitem{%
	\clearfield{issn}%
	\clearfield{isbn}%
	\clearlist{location}%
	\ifentrytype{book}%
		{%
		\clearfield{series}%
		\clearfield{volume}%
		\clearfield{pages}}%
		{%
		\clearname{editor}}%
}

\renewbibmacro*{doi+eprint+url}{%
  \newunit\newblock
  \iftoggle{bbx:doi}
    {\printfield{doi}}
    {}%
  \iftoggle{bbx:eprint}
    {\iffieldundef{doi}{\usebibmacro{eprint}}{}}
    {}%
  \newunit\newblock
  \iftoggle{bbx:url}
    {\iffieldundef{doi}{\iffieldundef{eprint}{\usebibmacro{url}}}} %
    {}}

\AtBeginBibliography{\footnotesize}	%

\usepackage{url} %
\makeatletter
\g@addto@macro{\UrlBreaks}{\UrlOrds}
\makeatother

\usepackage{hyperref}
\usepackage{csquotes}

\hypersetup{hidelinks,colorlinks=true,breaklinks=true,urlcolor=cit,citecolor=dkblue,linkcolor=dkblue}

\newtheorem{bigthm}{Theorem}

\newtheorem{theorem}{Theorem}[section]
\newtheorem{lemma}[theorem]{Lemma}

\newtheorem{proposition}[theorem]{Proposition}

\newtheorem{conjecture}[theorem]{Conjecture}
\newtheorem{corollary}[theorem]{Corollary}

\theoremstyle{definition}

\newtheorem{definition}[theorem]{Definition}
\newtheorem{example}[theorem]{Example}
\newtheorem{remark}[theorem]{Remark}

\newcommand{\R}{\mathbb{R}}

\addglobalbib{main} %

\usepackage[foot]{amsaddr}

\title[Single-shelf shuffling]{On the position matrix of single-shelf shuffle and card guessing}

\author{Raghavendra Tripathi}
\address{Division of Science, New York University, Abu Dhabi}
\email{r.tripathi@nyu.edu}

\date{} %


\begin{document}

\begin{abstract}
Mechanical shufflers used in many casinos employ a card shuffling scheme called  \emph{shelf shuffling}. In a single-shelf shuffling, cards arrive sequentially, and each incoming card is independently placed on the top or the bottom of a shelf with equal probability. The position matrix of a single-shelf shuffling encodes the probability that the $i$-th incoming card is in position $j$ after one round of single-shelf shuffle. The spectral properties of the position matrix of card shuffling schemes are helpful in the analysis of card guessing games without feedback. In this paper, we determine the full spectrum and the corresponding eigenspaces of the position matrix $M$ of a single-shelf shuffle. This strengthens and resolves two conjectures in a recent work [arXiv:2507.10294]. As a consequence of these results, we show that the maximum number of expected correct guesses without feedback after $k\geq (1+\epsilon)$ many shuffles is of the order $1+O(n^{-2\epsilon})$. On the other hand, the expected number of correct guesses after one shuffle is at most $\sqrt{2n/\pi}+1+O(n^{-1/2})$, and we give a strategy (not optimal) that achieves $\sqrt{2n/\pi}-1$ number of correct cards in expectation. 
\end{abstract}

\maketitle

\section{Introduction}
\label{sec:intro}
The analysis of card shuffling methods has occupied probabilists for over a century. Understanding the properties of various shufflings has important practical applications, and their analysis often reveals a rich interplay between probability, combinatorics, number theory, and representation theory. It is impossible to list all the relevant work in this area; we refer the reader to the recent book by Diaconis and Fulman~\cite{Diaconis23Math} and the extensive bibliography in the book for an introduction as well as a comprehensive account of this fascinating field. Arguably, the most important question about a shuffling scheme is how well it shuffles the cards, or equivalently, how many times one needs to shuffle the deck to get a truly random deck. A natural way to measure this is in terms of the mixing time for the shuffling scheme. The modern analysis of card shuffling began with rigorous analysis of mixing time of riffle-shuffle~\cite{Diaconis81Generating,Aldous83,aldous1986shuffling,bayer1992trailing,diaconis1995riffle}. Mixing time and the cutoff phenomenon for various shufflings continue to be studied~\cite{Ghosh20Total,Ghosh25Limit,Teyssier20LimitProfile,Bernstein19Cutoff,chen2025cutoff, Nestoridi25Cutoff,Sellke22Cutoff}. Alternatively, one can study various statistics like descents, valleys, increasing subsequences, etc., associated with the permutations induced by a shuffling scheme. Comparing these statistics with the corresponding statistics in a uniform random permutation also provides a qualitative and useful measure of the quality of a shuffling, and it has been carried out in several papers~\cite{clay2025limit,diaconis1995riffle,Kim21CLT}.

In this paper, we are concerned with the \emph{shelf shuffling} with a single-shelf $m=1$ (see Definition~\ref{def:ShelfShuffle}).  
\begin{definition}[$m$-Shelf shuffle]
\label{def:ShelfShuffle}
Consider a deck of $n$ cards arranged in increasing order. A shuffler has $m$-shelves. Cards are drawn from the bottom of the deck and placed on a shelf uniformly at random. Once a shelf is selected, the card is placed at the top or bottom of the pile with probability $1/2$ each. 
\end{definition}

The analysis of the shelf-shuffling first appeared in an influential paper by Diaconis, Fulman, and Holmes~\cite{DFH}. Among several other things,~\cite{DFH} derived a closed-form expression for the transition kernel of $m$-shelf shuffle and used it to compute the asymptotic rate of convergence to the stationary distribution in $\ell_{\infty}$ and the separation distance. In a recent work, Ottolini and Chen~\cite{chen2025cutoff} established the cutoff in total variation distance for $m$-shelf shuffling, Clay~\cite{clay2025limit} studied limit theorems for descents and inversions of shelf shuffles, and~\cite{Clay} studied the card guessing games with and without feedback for a single-shelf shuffler. 

Before we proceed further, we describe the card-guessing games with varying feedback. Consider a shuffled deck of $n$ cards. A player tries to guess the cards sequentially from the top of the deck. After each guess, the top card is removed from the deck. We refer to this as \emph{no feedback card guessing}. If the player is shown the removed card after each guess, we refer to this as \emph{complete feedback card guessing}. The goal of the player is to maximize the expected number of correct guesses. It is easy to see that for a uniformly shuffled deck of cards, one can guess only $1$ card in expectation without any feedback and $O(\log n)$ cards in expectation with complete feedback. As suggested by Bayer and Diaconis~\cite{bayer1992trailing}, one can measure the quality of a shuffling by understanding how many cards one can guess correctly with complete feedback and without feedback. A guessing strategy is a choice of card $g_j$ that the player guesses at the $j$-th step. We refer to a strategy that maximizes the expected number of correct guesses (with or without feedback) as \emph{optimal strategy}. Let us emphasize that an optimal strategy need not be unique. 

Card guessing games have been studied in the context of clinical trials and ESP experiments. For a uniform random deck with cards appearing with multiplicities, the optimal strategy with and without feedback was established in~\cite{diaconis1981analysis}. Since then, several authors have investigated the expected number of correct cards under the optimal strategy~\cite{he2023card,OS23Guessing}, CLT for the number of correct guesses~\cite{OT24}. Similar investigations have been carried out for card guessing for various shuffling schemes with different types of feedback~\cite{Ciucu98Nofeedback,Kuba24,DGHS22Partial,krityakierne2022no,KSPTY}.

For the shelf shuffling, this question was proposed in~\cite[Section 5]{DFH}, and the authors conjectured an optimal strategy for guessing with complete feedback and reported the results of a Monte Carlo experiment to compute the distribution of the number of correct guesses with the conjectured strategy for a deck of $n=52$ cards with various values of $m$. The first rigorous result in this direction was obtained very recently by Clay~\cite{Clay} in the case of a \emph{single-shelf shuffler} ($m=1$). From hereon, we only consider the single-shelf shuffling unless stated otherwise. Let us define the the \emph{position matrix} $M$ of a single-shelf shuffle:
\[
M(i, j) \coloneqq \mathbb{P}(\text{Card }i \text{ appears at position }j \text{ after one single-shelf shuffle})\;, \quad 1\leq i, j\leq n\;.
\]
A simple combinatorial argument shows that the position matrix of a single-shelf shuffle is given by
\[ M(i, j) = \frac{1}{2^i}\left(\binom{i-1}{j-1}+\binom{i-1}{n-j}\right)\;,\quad 1\leq i, j\leq n\;.\]
In the above and throughout this paper, we use the convention that $\binom{0}{0}=1$ and $\binom{a}{b}=0$ for $a<b$. Using the explicit description of the position matrix $M$, Clay proved that, under the strategy given by~\cite{DFH}, the expected number of correct guesses, in the complete feedback case, is $3n/4$. The no-feedback case is significantly harder. In this case, Clay~\cite{Clay} conjectured an optimal strategy, and proved that the expected number of correct guesses without feedback is at most $O(\sqrt{n})$ \emph{under his conjectured strategy}. Clay~\cite{Clay} also made two conjectures about the position matrix of a single-shelf shuffle.
\begin{conjecture}{\cite[Conjecture 1 and 2]{Clay}}
\label{conj:Clay}
Let $n\geq 3$. 
\begin{enumerate}
    \item $(1/4, v)$ is an eigenvalue-eigenvector pair of $M$ with multiplicity at least one where $v\in \mathbb{R}^n$ such that 
\[ v(k) = n^2-3nk+\frac{3}{2}k(k+1)-1\;.\] 
\item The non-zero eigenvalues of $M$ are $2^{-i}$ where $0 \leq i\leq n-1$ is an even integer.
\end{enumerate}
\end{conjecture}
Both conjectures are based on extensive numerical simulations. However, the numerical methods are too inefficient to provide any reasonable guess for the remaining eigenvectors. The motivation for this conjecture comes from the work of Ciucu~\cite{Ciucu98Nofeedback}, who studied the guessing strategy after a large number of dovetail shuffles. Ciucu's work crucially relies on diagonalizing the position matrix of dovetail shuffling. 

\subsection{Contributions of this paper}
In this paper, we refine and prove a strengthening of Conjecture~\ref{conj:Clay}. In particular, we provide explicit formulas for all eigenvectors of $M$. We also prove explicit formulas for the dual eigenvectors of $M$ and use them to diagonalize $M=XDX^{-1}$ where $D$ is a diagonal matrix. Using this, we show that after $k>(1+\epsilon)\log n$ repeated shelf shuffles, the maximum number of correct guesses without feedback is bounded by $1+O(n^{-2\epsilon})$. En route, we also exhibit that Clay's conjectured strategy in the no-feedback case is \emph{not} optimal. However, we show that one can only guess $O(\sqrt{n})$ many cards \emph{under any strategy} without feedback. Furthermore, we describe a very simple strategy (not optimal) using which the expected number of correct guesses matches our upper bound up to an additive error of $1+O(n^{-2})$.  

\subsection{Results}
Our first result fully describes the spectrum of $M$ with an explicit description of the associated eigenspaces, confirming Conjectures~\ref{conj:Clay}. Before we state our main result, we need some notation. We use Knuth's notation $x^{\underline{k}}$ to denote the falling factorial $x^{\underline{k}}=x(x-1)\ldots (x-k+1)$. We interpret $x^{\underline{t}}=0$ for all $t>x$ and $x^{\underline{0}}=1$ for all $x\geq 0$. We use $B_k$ to denote the Bernoulli numbers, that is, $B_{k}$ is the coefficient of $t^{k}$ in $t/(e^{t}-1)$. We now state our main result. 

\begin{bigthm}
\label{thm:Main Theorem1}
Fix $n\geq 2$ and let $M$ be the position matrix as defined above. Then, 
\begin{enumerate}
    \item  $\operatorname{ker}(M)=\operatorname{span}\{\eta^{(j)}: 1\leq j\leq  \lfloor n/2\rfloor\}$ where $\eta^{(j)}:=e_j-e_{n-j+1}$, where $e_k$ is the standard $k$-th basis vector $e_k(j)=\delta_k(j)$ in $\R^n$. 
    \item For every even integer $0\leq i\leq n-1$, $2^{-i}$ is an eigenvalue of $M$ with multiplicity $1$.
    \item For every even integer $0\leq i\leq n-1$, the vector $\zeta^{(i)}$ is an eigenvector of $M$ with eigenvalue $2^{-i}$, where
    \[
    \zeta^{(i)}(k)= \sum_{t=0}^{i}w_t^{(i)}(k-1)^{\underline{t}}\;,\qquad 1\leq k\leq n
    \]
    and 
    \[w_t^{(i)} =2^{i-t}B_{i-t}\binom{i}{t}(n-t-1)^{\underline{i-t}} = 2^{i-t}B_{i-t}\binom{i}{t}\frac{(n-t-1)!}{(n-i-1)!}\;.\]
\end{enumerate}
\end{bigthm}

\begin{remark}
Some remarks are in order.
\begin{enumerate}
    \item Note that $w^{(i)}_t=0$ whenever $t>i$.
    \item Since $B_{0}=1$, $B_{1}=-\frac{1}{2}$ and $B_{2k+1}=0$ for all $k\geq 1$. In particular, $w^{(i)}_t=0$ for all odd $t$ except $t=i-1$.
     \item Note that $\zeta^{(0)}(k)=1$ for all $k$. This is in agreement with the fact that $M$ is a doubly stochastic matrix. 
     \item The smallest non-zero eignvalue of $M$ is $2^{-2m}$ where $m=\lfloor (n-1)/2\rfloor$.
\end{enumerate}
\end{remark}

Let $G$ denote the strategy of guessing the card $2j-1$ at location $j$ and $n-j+1$ for $1\leq j\leq \lceil n/2\rceil$. Our next result shows that this simple strategy is asymptotically almost optimal.

\begin{bigthm}
\label{thm: AlmostOptimal}
Consider a deck of $n$ cards that is shuffled once using a single-shelf shuffler. 
Let $G$ be the strategy described above, and let $H$ be an optimal strategy without feedback, and let $S_n(G)$ and $S_{n}(H)$ denote the number of correct guesses under $G$ and $H$, respectively.
Then, 
\[
\sqrt{\frac{2n}{\pi}}-1 \leq \mathbb{E}\left[S_n(G)\right]\leq \mathbb{E}\left[S_n(H)\right]\leq \sqrt{\frac{2n}{\pi}} + 1 + O(n^{-1/2})\;.
\]
\end{bigthm}

\begin{remark}
\begin{enumerate}
    \item This result improves and strengthens Clay~\cite[Theorem 1.3]{clay2025limit} in several respects. First of all, it proves $O(\sqrt{n})$ upper bound on the optimal score in the no-feedback setting \emph{unconditionally}. Secondly, while Clay remarked that $\sqrt{\frac{2n}{\pi}}+ 0.68$ seems to be in strong agreement with the optimal score obtained numerically, the bound in~\cite{Clay} has an error term that is of order $O(\sqrt{n})$ itself. Thus, Theorem~\ref{thm: AlmostOptimal} refines ~\cite[Theorem 1.3]{clay2025limit}. 

    \item Optimal strategy for card guessing without feedback is to guess a card $g_j\in \arg\max \{M(i, j): 1\leq i\leq n\}$ at position $j$ and the expected number of correct guesses in this case is $\sum_{j=1}^{n}m_j$ where $m_j=\max_{1\leq i\leq n} M(i, j)$. Clay made an elaborate conjecture~\cite[Appendix A]{Clay} for $g_j$s. Let $m=\lceil \sqrt{n-1}/2-1\rceil$. Clay's conjectured strategy implies, in particular, that $m_j = M(2j-1, j)$ for all $j\leq \lfloor n/2\rfloor -m-1$, even though $g_j$ need not be unique. Unfortunately, the conjectured strategy is not optimal. It can be numerically verified that for $n=24$ and $j=10$ we have $ M(19, 10)< 0.098< 0.099 < M(20, 10) $. It is worth remarking that Clay's upper bound of $O(\sqrt{n})$ is proved under the assumption that this strategy is optimal. 
    %\item Finally, let us remark that following our proof, it is in fact possible to show that Clay's strategy achieves a larger number of correct guesses in expectation than our strategy. However, in the light of Theorem~\ref{thm: AlmostOptimal}, we know that it cannot do much better than one card. 
\end{enumerate}
\end{remark}

It is natural to ask, what is the maximum number of cards one can guess, in expectation, without feedback, if a deck of $n$ cards is shuffled $k$ times using a single-shelf shuffler? Let us denote this number by $E_{n, k}$. Our next result shows that $E_{n, k}\approx 1$ if $k> \log n$. This is analogous to the main result of Ciucu~\cite{Ciucu98Nofeedback} who performed a similar analysis for the dovetail shuffle.

\begin{bigthm}
\label{thm: Ciucu}
Let $E_{n, k}$ be as above. Let $\epsilon>0$ and let $k\geq (1+\epsilon)\log n$. Then, there exists $N_{\epsilon}$ and a universal constant $C$ such that 
\[ |E_{n, k}-1|\leq C n^{-2\epsilon},\]
whenever $n\geq N_{\epsilon}$. 
\end{bigthm}

%\subsection{Outline of the paper}
\section{Proof of Theorem~\ref{thm:Main Theorem1}:}

\begin{proof}[Proof of Theorem~\ref{thm:Main Theorem1}(1)]
Let $L$ be the $n\times n$ lower triangular matrix with entries 
$$ L(i, j)= \frac{1}{2^{i}}\binom{i-1}{j-1}.$$ 
Note that the diagonal entries of $L$ are $L(i, i)=2^{-i}$ for all $1\leq i\leq n$. Therefore, $L$ is invertible. Let $P$ be the $n\times n$ permutation matrix given by $P(i, j)=\delta_{i, n-i+1}$. Notice that
 $$(LP)_{i, j}= L_{i, n-j+1}=\frac{1}{2^i}\binom{i-1}{n-j}.$$
The position matrix $M$ of a single-shelf shuffle is given by $M=L+LP=L(I+P)$. Note that $L$ is lower triangular with diagonal entries $2^{-i}$, it follows that $L$ is invertible. Therefore, $\operatorname{ker}(M)=\operatorname{ker}(I+P)$. This immediately gives us the first part of Theorem~\ref{thm:Main Theorem1} by observing that $\operatorname{ker}(M)= \operatorname{span}\left\{\eta^{(j)}: 1\leq j\leq \lfloor n/2\rfloor \right\}$.    
\end{proof}

We now exhibit an eigenbasis $\mathcal{B}$ of $L$ and show that with respect to this eigenbasis, $(I+P)$ is upper triangular. In particular, $M$ is upper triangular with respect to the basis $\mathcal{B}$. This immediately gives the spectrum of $M$.

%We, then, explicitly compute the upper triangular matrix $T=[M]_{\mathcal{B}}^{\mathcal{B}}$ corresponding to $M$ in the basis $\mathcal{B}$, in other words, $M=BTB^{-1}$. We then find the eigenvalues and eigenvectors of $T$ (and the corresponding dual eigenvectors). This yields the eigenvalues and the corresponding eigenvectors of $M$. Here, we should remark that the matrix $B$ and its inverse $B^{-1}$ are explicit and simple to describe, which is crucial to us. 

\subsection{An eigenbasis of \texorpdfstring{$L$}{L}}
For $0\leq j\leq n-1$, define a vector $v^{(j)}$ such that 
\[v^{(j)}(k) = (k-1)^{\underline{j}}.\]
Recall our notation for the falling factorial $x^{\underline{m}}=x(x-1)\cdots (x-m+1)$ and the convention that $a^{\underline{0}}=1$ for any non-negative integer $a\geq 0$. In particular, $v^{(0)}$ is the vector whose coordinates are all one. 
\begin{proposition}
\label{prop:Leigenvector}
    For each $0\leq j\leq n-1$, the vector $v^{(j)}$ is an eignevector of $L$ with eigenvalue $2^{-j-1}$. 
\end{proposition}
\begin{proof}
Note that $L$ is a lower triangular matrix and $L(i, i)=2^{-i}$. In particular, the set of eigenvalues of $L$ is $\{2^{-1-j}: 0\leq j\leq n-1\}$. Fix $0\leq j\leq n-1$ and note that 
    \begin{align*}
        (Lv^{(j)})_{i} = 2^{-i}\sum_{k=1}^{i}\binom{i-1}{k-1}(k-1)^{\underline{j}}=2^{-1}\;\mathbb{E}[X^{\underline{j}}]=2^{-j-1}(i-1)^{\underline{j}}\;,
    \end{align*}
    where $X$ is a Binomial random variable $X\sim \mathrm{Bin}(i-1, 1/2)$. And the last equality follows from the well-known identity for the factorial moment of a Binomial random variable
    \[
    \mathbb{E}\left[\mathrm{Bin}(n, p)^{\underline{m}}\right] = p^{m}n^{\underline{m}}\;.
    \]
    %\[\mathbb{E}[X^{\underline{j}}]=2^{-(i-1)}(i-1)^{\underline{j}}\;.\]
    We conclude that $Lv^{(j)}=2^{-j-1}v^{(j)}$ for each $0\leq j\leq n-1$. This concludes the proof.
    
\end{proof}
Let $B$ be the $n\times n$ matrix whose columns are eigenvectors of $L$, that is, 
\[ B = \begin{bmatrix}
    v^{(0)}| & v^{(1)}| & \cdots| & v^{(n-1)}
\end{bmatrix}\;.\]
It follows from Proposition~\ref{prop:Leigenvector} that $B$ is invertible and $L=B\widetilde{D}B^{-1}$ where $\widetilde{D}$ be the diagonal matrix 
\[\widetilde{D} = \operatorname{diag}[2^{-1}, 2^{-2}, \ldots, 2^{-(n-1)}]\;.\] 
It will be useful later to have an explicit formula for $B^{-1}$ given in the following lemma, which can be verified directly. Alternatively, we note that $B^{-1}(k, \ell)$ is the coefficient of $x^{\ell-1}$ in the polynomial $x^{\underline{k-1}}$. 

\begin{lemma}\label{lem:B-inverse}
The matrix $B$ is invertible. Its inverse $B^{-1}$ is lower triangular with
\begin{equation*}\label{eq:B-inverse}
  (B^{-1})(k,\ell)
  =
  \begin{cases}
    \displaystyle
    (-1)^{\,k-\ell}\,\frac{1}{(k-1)!}\binom{k-1}{\ell-1},
      & 1\le \ell\le k\le n,\\[6pt]
    0,& \ell>k.
  \end{cases}
\end{equation*}
\end{lemma}

\subsection{\texorpdfstring{$I+P$}{I plus P} is upper triangular in the basis \texorpdfstring{$\mathcal{B}$}{B}}
Let $\mathcal{B}\coloneqq \{v^{(j)}: 0\leq j\leq n-1\}$ be the eigenbasis of $L$ which we will refer as the \emph{falling factorial basis}. We now show that $(I+P)$ (and therefore $M=L(I+P)$) is upper triangular in the basis $\mathcal{B}$. To this end, set 
\[W_{j} :=\mathrm{span}\{v^{(k)}: 0\leq k\leq j\}, \quad 0\leq j\leq n-1\;.\] 
We will show that $(I+P)Lv^{(j)}\in W_j$ for each $0\leq j\leq n-1$. We need the following Lemma. 

\begin{lemma}
\label{lemma:Polynomial basis identity}
    Let $N, m$ be a fixed positive integer. Then, 
    \[(N-x)^{\underline{m}} = \sum_{k=0}^{m}c_{N}(k, m) x^{\underline{k}},\]
    where \[c_N(k, m) = (-1)^{k}(N-k)^{\underline{m-k}}\binom{m}{k} = (-1)^{k}\frac{(N-k)!}{(N-m)!}\binom{m}{k},\qquad 0\leq k\leq m\;.\]
    In particular, $c_{N}(m, m)=(-1)^{m}$ and $ c_N(0, m)=N^{\underline{m}}$ for all $m$. Furthermore, $c_N(k, m)=0$ for $k>m$.
\end{lemma}
\begin{proof}
    The proof follows from induction. We give the details for completeness. Fix $N\geq 1$. For $m=0$ and $m=1$, we easily verify that $c_N(0, 0)=1, \quad c_N(0, 1)=1, \quad c_N(1, 1)=-1$. Assume that the formula holds for every non-negative integer smaller $m$ for some $m\geq 1$. We now observe that 
    \begin{align*}
        (N-x)^{\underline{m+1}} &= (N-x-m)(N-x)^{\underline{m}}=(N-x-m)\sum_{k=0}^{m}c_N(k, m)x^{\underline{k}}\\
        &= \sum_{k=0}^{m}(N-m-k)c_N(k, m)x^{\underline{k}}-\sum_{k=0}^{m}c_N(k, m)(x-k)x^{\underline{k}}\\
        &= \sum_{k=0}^{m}(N-m-k)c_N(k, m)x^{\underline{k}}-\sum_{k=0}^{m}c_N(k, m)x^{\underline{k+1}}\;.
        %&= (N-m)c_N(k, m)+\sum_{k=1}^{m}\Big((N-m+k)c_N(k, m)-c_N(k-1, m)\Big)-c_N(m, m)x^{m+1}\;.
    \end{align*}
    Comparing the coefficients of $x^{\underline{k}}$ on both sides, we see that 
    \begin{align*}
         c_{N}(0, m+1) &= (N-m)c_N(0, m)\,,\\ 
         c_{N}(k, m+1) &= (N-m-k)c_N(k, m)-c_N(k-1, m)\,,\\
         c_N(m+1, m+1) &=-c_N(m, m)\;.
    \end{align*}
    Plugging the $c_N(0, m)=N^{\underline{m}}$ and $C_N(m, m)=(-1)^m$, we easily see that $c_N(0, m+1)=N^{\underline{m+1}}$ and $c_N(m+1, m+1)=(-1)^{m+1}$ holds. 
    Now assume that $1\leq k\leq m$. It suffices to show that 
    \[
\frac{(N-k)!}{(N-m-1)!}\binom{m+1}{k} = (N-m-k)\frac{(N-k)!}{(N-m)!}\binom{m}{k}+\frac{(N-k+1)!}{(N-m)!}\binom{m}{k-1}\;.
    \]
 Simplifying both sides, this is equivalent to  
 \[(N-m)\binom{m+1}{k} = (N-m-k)\binom{m}{k}+(N-k+1)\binom{m}{k-1}\;.\]
 We now use the Pascal's identity $\binom{m+1}{k}=\binom{m}{k}+\binom{m}{k-1}$ to conclude that right hand side is equal to
 \begin{align*}
    (N-m-k)\binom{m+1}{k}+(m+1)\binom{m}{k-1}=(N-m)\binom{m+1}{k}-k\binom{m+1}{k}+(m+1)\binom{m}{k-1}\;.
 \end{align*}
 The proof is now complete by observing that
 \[
 (m-1)\binom{m}{k-1}=k\binom{m+1}{k}\;.
 \] 
   % We now use the identity $\binom{m}{k}=\binom{m+1}{k}-\binom{m}{k-1}$, we get that 
   % \[(N-k)^{\underline{(m+1)-k}}\binom{m}{k}=(N-k)^{\underline{(m+1)-k}}\binom{m+1}{k}-(N-k)^{\underline{(m+1)-k}}\binom{m}{k-1}\;.\]
\end{proof}

As an immediate consequence, we get the $n\times n$ matrix corresponding to the map $x\mapsto (I+P)x$ in the basis $\mathcal{B}$. It will be convenient for us to index the rows and columns of this matrix by $0\leq i, j\leq n-1$. 

\begin{proposition}[$(I+P)$ is upper triangular in basis $\mathcal{B}$]
For any $0\leq j\leq n-1$, we have
\[
Pv^{(j)}= (-1)^{j}v^{(j)}+ \sum_{k=0}^{j-1} c_{n-1}(k, j) v^{(k)}, 
\]
In particular, the operator $(I+P)$ in the basis $\mathcal{B}$ has the matrix representation 
\[[(I+P)]_{\mathcal{B}}^{\mathcal{B}}= [\delta_{j, k}+c_{n-1}(k, j)]_{0\leq j, k\leq n-1}\;.\]
\end{proposition}
\begin{proof}
Fix a coordinate $1\leq t\leq n$ and observe that 
\[ Pv^{(j)}(t) =(n-t)^{\underline{j}}= ((n-1)-(t-1))^{\underline{j}}\;.\]
Using Lemma~\ref{lemma:Polynomial basis identity}, with $N=n-1$, we conclude
\[ 
Pv^{(j)}(t) =  \sum_{k=0}^{j} c_{N}(k, j)(t-1)^{\underline{k}} =  (-1)^{j}v^{(j)}(t) + \sum_{k=0}^{j-1} c_{N}(k, j)v^{(k)}(t)\;. 
\]
Since the above identity holds for each coordinate $t$ and the coefficients $c_{N}(k, j)$ are independent $t$, we conclude that  
\[
Pv^{(j)}= \sum_{k=0}^{j} c_{n-1}(k, j) v^{(k)}\;.
\]
Thus, we obtain $[P]_{\mathcal{B}}^{\mathcal{B}}(k, j)=c_{n-1}(k, j)$ for $0\leq k, j, \leq n-1$.
\end{proof}

\begin{corollary}[$M$ is upper triangular in the basis $\mathcal{B}$]
\label{cor:FormylaT}
    Let $T=[M]_{\mathcal{B}}^{\mathcal{B}}$ be the matrix of the operator $M$ in the basis $\mathcal{B}$. For $0\leq j\leq n-1$ 
    \[T(j, j) = 2^{-j-1}(1+(-1)^{j}) =\begin{cases}
        2^{-j} &  \text{if } j \text{ is even},\\
        0 & \text{otherwise}.
    \end{cases}\]
    And, for $0\leq j< k\leq n-1$, we have
    \[ T(j, k) = 2^{-1-j}c_{n-1}(j, k) =  2^{-1-j}(-1)^{j}(n-1-j)^{\underline{k-j}}\binom{k}{j}=2^{-1-j}(-1)^j\frac{(n-1-j)!}{(n-1-k)!}\binom{k}{j}\;.\]
\end{corollary}

This determines the spectrum of the position matrix $M$, confirming~\cite[Conjecture 3.2]{Clay}.
\begin{corollary}[Spectrum of $M$ (Theorem~\ref{thm:Main Theorem1} (2))]
    The spectrum of the position matrix $M$ of a single-shelf shuffle is given by 
    \[\{0\}\cup\{2^{-2i}: i\in \mathbb{Z}, \;\; 0\leq 2i\leq n-1\}\;.\]
    Furthermore, each non-zero eigenvalue $2^{-2i}$ occurs with multiplicity $1$. 
\end{corollary}

\subsection{Non-trivial eigenvectors of \texorpdfstring{$T$}{T}}
Recall that we index the rows and columns of $T$ by $0\leq i\leq n-1$. It will be simpler for us to index the coordinates of a vector $x\in \mathbb{R}^n$ also by $0\leq i\leq n-1$. Recall from the definition of $T$ that $T(0, 0)=1$. Since $T$ is upper triangular, it is clear that $w^{(0)}\coloneqq (1, 0, \ldots, 0)$  is an eigenvector of $T$ with eigenvalue $1$. We now construct the eigenvectors $w^{(i)}$ of $T$ corresponding to the eigenvalue $2^{-i}$ in general. Let $0\leq i\leq n-1$ be an even integer. Let us define $w^{(i)}$ so that
\[ w^{(i)}_{i+1}=\cdots = w^{(i)}_{n-1}=0\;,\]
and for $0\leq t\leq i$, let 
\begin{equation}
\label{eqn:EigenvectorDefn}
    w^{(i)}_t = 2^{i-t}B_{i-t}\binom{i}{t}\frac{(n-t-1)!}{(n-i-1)!}\;,
\end{equation}
where $B_{k}$ denotes the $k$-th Bernoulli number. Note that the definition $w^{(i)}_t$ in~\eqref{eqn:EigenvectorDefn} is valid for all $0\leq t\leq n-1$, if we interpret $\binom{i}{t}=0$ for $t>i$.

\begin{proposition}
\label{prop:EigenvaluesT}
For every even integer $0\leq i\leq n-1$, the vector $w^{(i)}$ defined in~\eqref{eqn:EigenvectorDefn}
 is an eigenvector of $T$ with eigenvalue $2^{-i}$. 
 \end{proposition}

 We need the following lemma, which follows from the elementary properties of the Bernoulli polynomials. We skip the proof of this lemma.
\begin{lemma}
\label{lem:IdentityBernoulli}
 \[
        \sum_{k=0}^{n}\binom{n}{k}B_{k}2^k =(2-2^n)B_n\;.
\]
\end{lemma}
\begin{proof}[Proof of Proposition~\ref{prop:EigenvaluesT}]
We know that $w^{(0)}$ is an eigenvector with eigenvalue $1$. Fix an even integer $i\geq 2$, and we will verify that $(Tw^{(i)})_k=2^{-i}w^{(i)}_k$ for all $0\leq k\leq n-1$. Since $w^{(i)}_k=0$ for $t>i$, we trivially get  
\[(Tw^{(i)})_{k}=2^{-i}w^{(i)}_{k}=0, \quad k> i\;.\]
Also, since $T(i, i)=2^{-i}$, we have $(Tw^{(i)})_i = T(i, i) = 2^{-i}=2^{-i}w^{(i)}_i$. Furthermore, since $T(i-1, i-1)=0$ we obtain 
\[(Tw^{(i)})_{i-1} = T(i-1, i)w^{(i)}_i = -2^{-i}i(n-i)=2^{-i}w^{(i)}_{i-1}\;.\]
For $k<i-1$, we separately deal with two cases.

\subsection*{\texorpdfstring{Odd $k<i-1$}{Odd k less than i-1}}
Fix an odd integer $k<i-1$, and observe that 
\begin{align*}
    (Tw^{(i)})_{k} &= \sum_{m=k+1}^{i}T(k, m)w^{(i)}_m \\
    &= \sum_{m=k+1}^{i}2^{-1-k}(-1)^k\binom{m}{k}\frac{(n-1-k)!}{(n-1-m)!}2^{i-m}B_{i-m}\binom{i}{m}\frac{(n-m-1)!}{(n-i-1)!}\\
    &= -2^{-1-k}\frac{(n-k-1)!}{(n-i-1)!}\sum_{m=k+1}^{i}2^{i-m}B_{i-m}\binom{m}{k}\binom{i}{m}\\
    &= -2^{-1-k}\frac{(n-k-1)!}{(n-i-1)!}\sum_{m=k+1}^{i}2^{i-m}B_{i-m}\binom{m}{k}\binom{i}{m}
    ;.
\end{align*}
Setting $t=i-m$ and $N=i-k$, and reindexing we can rewrite
\[
\sum_{m=k+1}^{i}2^{i-m}B_{i-m}\binom{m}{k}\binom{i}{m}= \binom{i}{k}\sum_{t=0}^{N-1}2^{t}B_t\binom{N}{t}=\binom{i}{k}\sum_{t=0}^{N}2^{t}B_t\binom{N}{t},
\]
where the last equality is true bacause $N=i-j>1$ is odd and hence $B_{N}=0$. We now use Lemma~\ref{lem:IdentityBernoulli} and to conclude that $\sum_{t=0}^{N}2^{t}B_t\binom{N}{t}=(2-2^N)B_N$ which is zero because $N>1$ is odd. Therefore, we conclude that $(Tw^{(i)})_{k}=0$.

\subsection*{Even \texorpdfstring{$k<i$}{k less than i}}
Now fix an even integer $0<k<i$, and observe that 
\[
(Tw^{(i)})_{k} = T(k, k)w^{(i)}_k+\sum_{m=k+1}^{i}T(k, m)w^{(i)}_m\;.
\]
Let us define
\[ P_n(i, k) = \frac{(n-i-1)!}{(n-k-1)!}\binom{i}{k}\]
for notational simplicity. Arguing as previously (but noticing that $(-1)^k=1$ since $k$ is even), and setting $N=i-j$ we obtain 
\begin{align*}
   \sum_{m=k+1}^{i}T(k, m)w^{(i)}_m  &=  2^{-1-k}P_n(i, k)\sum_{t=0}^{N-1}2^{t}B_t\binom{N}{t},\\
   &= 2^{-1-k}P_n(i, k)\left(\sum_{t=0}^{N}2^{t}B_t\binom{N}{t} -2^{N}B_{N}\right),\\
   %&= 2^{-1-k}P_n(i, k)\left(2^{N}B_N(1/2)-2^{N}B_N\right)\\
   &=2^{-k}P_n(i, k)(1-2^N)B_N\;,
   %= 2^{-k}(1-2^{i-k}) B_{i-k} P_n(i, k).
\end{align*}
where we used Lemma~\ref{lem:IdentityBernoulli} in the last two lines. Finally, note that 
\[T(k, k)w^{(i)}_k = 2^{-k}2^{i-k}B_{i-k}P_n(i, k)\;.\]
Combining everything we obtain 
\[(Tw^{(i)})_k = B_{i-k}P_n(i, k)\Big( 2^{-k}(1-2^{i-k})+2^{-k}2^{i-k}) \Big) = 2^{-k}B_{i-k}P_n(i, k)\;.\]
The proof is complete by observing that 
\[2^{-i}w^{(i)}_k = 2^{-i}2^{i-k}B_{i-k}P_n(i, k) = 2^{-k}B_{i-k}P_n(i, k)\;,\]
and therefore $(Tw^{(i)})_{k} = 2^{-i}w^{(i)}_k$.     
\end{proof}

\begin{proof}[Proof of Theorem~\ref{thm:Main Theorem1}(3)]
Recall that $T$ is the matrix representation $x\mapsto Mx$ with respect to the falling factorial basis $\mathcal{B}$. In particular, $T$ and $M$ have the same spectrum. Furthermore, the eigenvectors of $M$ corresponding to the eigenvalue $2^{-i}$ is given by 
\[\zeta^{(i)} = Bw^{(i)}\;.\]
This completes the proof of Theorem~\ref{thm:Main Theorem1} by using the definition of $B$. 

\end{proof}
We illustrate a simple corollary of Theorem~\ref{thm:Main Theorem1} that confirms~\cite[Conjecture 1]{Clay}.
\begin{example}
Let $i=2$ and consider the eigenvalue $2^{-2}=1/4$ of $M$. The eigenvectors of $T$ are given by
    \[w^{(2)} = (x_0, x_1, x_2, 0, 0, \ldots, 0),\]
    where 
    \[ x_0=\frac{2}{3}(n-2)(n-1), \quad x_1=-2(n-2), \quad x_2=1 \;.\]
    In particular, the eigenvector of the position matrix $M$ in the standard basis is given by $v = x_0v^{(0)}+x_1v^{(1)}+v^{(2)}$. More explicitly, we obtain
    \begin{align*}
        v_{k} &= \frac{2}{3}(n-2)(n-1)-2(n-2)(k-1)+(k-1)(k-2)\\
        &= \frac{2}{3}\left(n^2-3nk+\frac{3}{2}k(k+1)-1\right).
    \end{align*}
    This is the eigenvector conjectured by Clay~\cite[Conjecture 1]{Clay} up to the scalar factor $2/3$.
\end{example}

\section{Proof of Theorem~\ref{thm: Ciucu}}
\label{sec:ProofIII}

\subsection{Diagonalization of \texorpdfstring{$M$}{M}}
Let $\mathcal{E}_{n}=\{j: 0\leq j\leq n-1,\; j\text{ is even}\}$. For each $j\in \mathcal{E}_n$, we define the (suitably normalized) left-eigenvectors $\widetilde{\zeta}^{(j)}$ of $M$  so that 
\[(\widetilde{\zeta}^{(i)})^{\top}M=2^{-i}(\widetilde{\zeta}^{(i)})^{\top}, \qquad  \text{and}\qquad 
\langle \widetilde{\zeta}^{(i)}, \zeta^{(j)}\rangle =\delta_{i, j} \text{ for all } i, j\in \mathcal E_n\;.\]
Then, we have the following eigendecomposition of $M$
\[
M = \sum_{j\in \mathcal{E}_n} 2^{-j}\; \zeta^{(j)}\otimes\widetilde{\zeta}^{(j)}\;,
\]
where $u\otimes v$ denotes the outer product of two vectors. Since $M$ is bistochastic, we know that $\widetilde{\zeta}^{(0)}=\frac{1}{n}(1, \cdots, 1)$. The normalization $\frac{1}{n}$ is chosen so that $\langle \widetilde{\zeta}^{(0)}, \zeta^{(0)}\rangle =1$. Using this, we conclude that 
\begin{equation}
\label{eqn:M^kExapnsion}
    M^{k} = \frac{1}{n}J + \sum_{j\in \mathcal{E}_n^*}2^{-jk} \zeta^{(j)}\otimes\widetilde{\zeta}^{(j)}\;.
\end{equation}

Our first step is to explicitly describe the left-eigenvectors $\widetilde{\zeta}^{(j)}$ of $M$. Recall that the matrix $T=[M]_{\mathcal{B}}^{\mathcal{B}}$. Our first lemma describes the (normalized) left-eigenvectors of $T$. It will be convenient to index the coordinates of the vectors by $0\leq t\leq n-1$.  
\begin{lemma}[Left-eignevctors of $T$]
Fix $i\in \mathcal{E}_n$. For $0\leq t\leq n-1$
\begin{equation}\label{eq:u-left-nonzero}
  \widetilde{w}^{(i)}_t
  =
  \begin{cases}
    0, & t<i,\\[4pt]
    \displaystyle
    \frac{1}{t-i+1}\binom{t}{i}\,\frac{(n-1-i)!}{(n-1-t)!},
      & t\ge i.
  \end{cases}
\end{equation}
Then, $(\widetilde{w}^{(i)})^{\top}T=2^{-i}(\widetilde{w}^{(i)})^{\top}$. Furthermore, $\langle \widetilde{w}^{(i)}, {w}^{(j)}\rangle=\delta_{i, j}$ for any $i, j\in \mathcal{E}_n$. 
\end{lemma}
\begin{proof}
Fix $i\in \mathcal{E}_n$ and $0\leq u\leq n-1$. It is immediate from~\eqref{eq:u-left-nonzero} and the description of $T$ in Corollary~\ref{cor:FormylaT} that
 \[ ((\widetilde{w}^{(i)})^{\top}T)_{u} =\begin{cases}
     0, &\quad  u<i,\\
     2^{-i}=2^{-i}w^{(i)}_i, &\quad u=i, \\
    \sum_{t=i}^{u} \widetilde{w}^{(i)}_t T(t, u), & \quad u>i\;.
 \end{cases}\]
 Now assume that $u>i$. For $t\leq u-1$, we have
\begin{align*}
   \widetilde{w}^{(i)}_t T(t, u) %&=  \frac{1}{t-i+1}\binom{t}{i}\frac{(n-1-i)!}{(n-1-t)!}\;2^{-1-t}(-1)^t \frac{(n-1-t)!}{(n-1-u)!}\binom{u}{t}\\
   &= 2^{-1-t}(-1)^t\frac{1}{t-i+1}\frac{(n-1-i)!}{(n-1-u)!}\binom{t}{i}\binom{u}{t}\\
   &= 2^{-1-t}(-1)^t\frac{1}{t-i+1}\frac{(n-1-i)!}{(n-1-u)!}\binom{u}{i}\binom{u-i}{t-i}\;.
\end{align*}
And, 
\[ \widetilde{w}^{(i)}_u T(u, u)=\begin{cases}
    0, & \qquad u \text{ is odd}, \\
    \frac{2^{-u}}{u-i+1}\binom{u}{i}\,\frac{(n-1-i)!}{(n-1-u)!}, &\qquad u \text{ is even}.
\end{cases}\]
Summing over $t$, and comparing with $2^{-i}\widetilde{w}^{(i)}_u$, we conclude that it suffices to show that 
\begin{equation}
\label{eq:NeedEigenvector}
 \sum_{t=i}^{u-1}\frac{2^{-1-t}(-1)^t}{t-i+1}\binom{u-i}{t-i}+ \frac{2^{-u}}{u-i+1}\mathrm{1}\{u \text{ is even}\} = \frac{2^{-i}}{u-i+1}\;.   
\end{equation}
Using the fact that $(-1)^i=i$ since $i$ is even, and reindexing the sum, and writing $m=u-i$,~\eqref{eq:NeedEigenvector} is equivalent to 
\begin{equation}
\label{eq:FindalNeeded}
    \frac{m+1}{2}\sum_{r=0}^{m-1}\binom{m}{r}\frac{(-1)^r}{2^r}\frac{1}{r+1}
+\mathbf 1_{\{m\ \mathrm{even}\}}2^{-m}
=1
\end{equation}

To this end, we observe that 
\begin{align*}
  \sum_{r=0}^{m-1}\binom{m}{r}\frac{(-1)^r}{2^r}\frac{1}{r+1} &= \int_0^1\Big(1-\frac{x}{2}\Big)^m dx-\int_0^1\Big(-\frac{x}{2}\Big)^m dx\\
  &=\frac{2\big(1-2^{-(m+1)}\big)}{m+1}-\frac{(-1)^m 2^{-m}}{m+1}.
\end{align*}
In particular, the LHS in~\eqref{eq:FindalNeeded} simplified to 
\[1-2^{-(m+1)}(1+(-1)^m)+2^{-m}\;\mathrm{1}\{m\ \mathrm{even}\}\;,\]
which is always $1$. This proves that $\widetilde{w}^{(i)}$ is a left-eigenvector of $M$ with eigenvalue $2^{-i}$. 

For any $i, j\in \mathcal{E}_n$, note that by definition $Tw^{(j)}=2^{-j}w^{(j)}$ and $T^{\top}\widetilde{w}^{(i)}=2^{-i}\widetilde{w}^{(i)}$. In particular, for $i\neq j$ we have
\[2^{-j}\langle \widetilde{w}^{(i)}, w^{(j)}\rangle=\langle \widetilde{w}^{(i)}, Tw^{(j)}\rangle = \langle T^{\top} \widetilde{w}^{(i)}, w^{(j)}\rangle=2^{-i}\langle \widetilde{w}^{(i)}, w^{(j)}\rangle\;. \]
Therefore, $\langle \widetilde{w}^{(i)}, w^{(j)}\rangle=0$. On the other hand, $\langle \widetilde{w}^{(i)}, w^{(i)}\rangle=\widetilde{w}^{(i)}_iw^{(i)}_i=1$. This completes the proof. 
\end{proof}

By a change of basis, it is immediate that $\widetilde\zeta^{(i)}=(\widetilde{w}^{(i)})^\top B^{-1}$ is a left-iegnevctor of $M$ with eigenvaleu $2^{-i}$ for $i\in \mathcal{E}_n$. Furthermore, it is easily checked that $\{\widetilde\zeta^{(i)}: i\in \mathcal{E}_n\}$ is dual to $\{\zeta^{(i)}: i\in \mathcal{E}_n\}$. We now compute $\widetilde\zeta^{(i)}(a)$ explicitly in a closed form, which will be useful later to get an $\ell_{\infty}$ bound on $\widetilde{\zeta}^{(i)}$.

\begin{theorem}[Explicit left eigenvectors of $M$]
\label{lem:explicit-left-M}
Let $n\ge 2$ and let $i\in \mathcal E_n$. For each
$1\le a\le n$, the $a$th coordinate of the left eigenvector
$\widetilde\zeta^{(i)}=(\widetilde{w}^{(i)})^\top B^{-1}$ is given by
\begin{equation}\label{eq:zeta-tilde-explicit}
\widetilde\zeta^{(i)}(a)
= \frac{1}{i!(n-i)}\Bigl[
   (-1)^{\,i-a}\binom{i-1}{a-1}
 + (-1)^{\,n-1-a}\binom{i-1}{a-1-(n-i)}
 \Bigr],
\end{equation}
where, as usual, $\binom{m}{r}=0$ if $r<0$ or $r>m$.

\end{theorem}
\begin{proof}
For $i\in \mathcal{E}_n$ and $1\leq a\leq n$, we have
\[
\widetilde\zeta^{(i)}(a)
= \sum_{t=0}^{n-1} \widetilde{w}^{(i)}_t\, (B^{-1})(t+1,a).
\]
Since $w^{e,(i)}_t=0$ for $t<i$ and $(B^{-1})(t+1,a)=0$ for $t+1<a$, we get
\begin{equation}\label{eq:tilde-zeta-start}
\widetilde\zeta^{(i)}(a)
= \sum_{t=\max\{i,a-1\}}^{n-1}
   \frac{1}{t-i+1}\binom{t}{i}\frac{(n-1-i)!}{(n-1-t)!}
   \frac{(-1)^{t+1-a}}{t!}\binom{t}{a-1}.
\end{equation}

Set $N:=n-1$ and $m:=N-i=n-1-i$. Write $u:=a-1$.
We reindex the sum by $t=i+s$, where $0\le s\le m$. Note that
$(N-t)!=(m-s)!$ and
\[
\frac{1}{t-i+1}\binom{t}{i}\frac{1}{t!}
= \frac{1}{s+1}\cdot\frac{(i+s)!}{i!\,s!}\cdot\frac{1}{(i+s)!}
= \frac{1}{(s+1)\,i!\,s!}.
\]
Thus \eqref{eq:tilde-zeta-start} becomes
\[
\widetilde\zeta^{(i)}(a)
= \frac{(N-i)!}{i!}
  \sum_{s=\max\{0,u-i\}}^{m}
  \frac{(-1)^{i+s+1-a}}{(s+1)\,s!\,(m-s)!}\binom{i+s}{u}.
\]
Define
\[
S(u) := \sum_{s=0}^{m}
  \frac{(-1)^s}{(s+1)\,s!\,(m-s)!}\binom{i+s}{u}\;.
\]
Note that $\binom{i+s}{u}$ vanishes for $u>i+s$, therefore we can write
\[
\widetilde\zeta^{(i)}(a)
= (-1)^{i+1-u}\frac{(N-i)!}{i!}\,S(u),\qquad u=a-1.
\]
We evaluate $S(u)$ via a generating function. Using
$\binom{i+s}{u}=[x^u](1+x)^{i+s}$, we get
\begin{align*}
  S(u)
&= [x^u]\sum_{s=0}^{m}
   \frac{(-1)^s}{(s+1)\,s!\,(m-s)!}(1+x)^{i+s} \\ &=[x^u]\bigl[(1+x)^i G(1+x)\bigr],
\end{align*}
where
\begin{align*}
    G(t) &:= \sum_{s=0}^{m}\frac{(-1)^s t^s}{(s+1)\,s!\,(m-s)!}  = \int_0^1 \sum_{s=0}^{m}\frac{(-1)^s(ty)^s}{s!\,(m-s)!}\,dy\\
&= \frac{1}{m!}\int_0^1 (1-ty)^m\,dy = \frac{1}{m!}\frac{1 - (1-t)^{m+1}}{t(m+1)} = \frac{1}{(m+1)!}\cdot\frac{1-(1-t)^{m+1}}{t}.
\end{align*}

Now we evaluate $G$ at $t=1+x$:
\[
G(1+x)
= \frac{1}{(m+1)!}\cdot\frac{1-(1-(1+x))^{m+1}}{1+x}
= \frac{1}{(m+1)!}\cdot\frac{1-(-x)^{m+1}}{1+x}.
\]
Thus,
\[
(1+x)^i G(1+x)
= \frac{1}{(m+1)!}(1+x)^{i-1}\bigl(1-(-x)^{m+1}\bigr).
\]
Taking coefficients of $x^u$, yields
\[
S(u)
= \frac{1}{(m+1)!}\Bigl(
     [x^u](1+x)^{i-1}
   - [x^u](1+x)^{i-1}(-x)^{m+1}
   \Bigr).
\]
The two terms are
\[
[x^u](1+x)^{i-1} = \binom{i-1}{u},
\qquad
[x^u](1+x)^{i-1}(-x)^{m+1}
= (-1)^{m+1}\binom{i-1}{u-(m+1)},
\]
with the convention $\binom{i-1}{r}=0$ if $r<0$ or $r>i-1$. Therefore
\[
S(u)
= \frac{1}{(m+1)!}
   \Bigl[\binom{i-1}{u}
         - (-1)^{m+1}\binom{i-1}{u-(m+1)}\Bigr].
\]

Recall $N=n-1$ and $m=N-i=n-1-i$, so
\[
\frac{(N-i)!}{(m+1)!} = \frac{m!}{(m+1)!} = \frac{1}{m+1} = \frac{1}{n-i}.
\]
Putting everything together,
\[
\widetilde\zeta^{(i)}(a)
= (-1)^{i+1-u}\frac{(N-i)!}{i!}\,S(u)
= \frac{(-1)^{i+1-u}}{i!(n-i)}
  \Bigl[\binom{i-1}{u}
        - (-1)^{m+1}\binom{i-1}{u-(m+1)}\Bigr].
\]
Using $u=a-1$ and $m+1=n-i$, gives~\eqref{eq:zeta-tilde-explicit}.
\end{proof}

\subsection{Enrtywise bounds on \texorpdfstring{$\zeta^{(i)}$}{zeta(i)} and \texorpdfstring{$\widetilde{\zeta}^{(i)}$}{tilde zeta(i)}}
We will now establish the bounds on the $\ell_{\infty}$ norm of $\zeta^{(i)}$ and $\widetilde{\zeta}^{(i)}$ for $i\in \mathcal{E}_n^{*}:=\mathcal{E}_n\setminus \{0\}$. Before we state our first bound, we recall two standard facts about Bernoulli numbers (see, e.g.,~\cite{GKP} for proofs).
\begin{lemma}[Basic properties of Bernoulli numbers]\label{lem:Bernoulli-basic}
Let $(B_s)_{s\ge 0}$ be the Bernoulli numbers, defined by
\[
  \frac{z}{e^z - 1}
  = \sum_{s=0}^\infty B_s \frac{z^s}{s!},\qquad |z|<2\pi.
\]
Then:
\begin{enumerate}
  \item $B_0=1$, $B_1=-\frac12$, and $B_{2m+1}=0$ for all $m\ge 1$.
  \item For all integers $m\ge 1$,
  \[
    |B_{2m}|
    \le \frac{2(2m)!}{(2\pi)^{2m}}.
  \]
\end{enumerate}
\end{lemma}
We note the following immediate consequence of
Lemma~\ref{lem:Bernoulli-basic} for later use.
\begin{lemma}\label{lem:A-i-bound}
For $i\geq 2$, define \[
  A_i := |B_0| + 2|B_1|\binom{i}{1}
        + \sum_{\substack{s\ge 2\\ s\le i}} 2^{s}|B_s|\binom{i}{s}
      = 1 + i + \sum_{s=2}^{i} 2^{s}|B_s|\binom{i}{s}.
\]
For every integer $i\ge 2$,
\[
  A_i \le 4\,i!.
\]
\end{lemma}
\begin{proof}
By Lemma~\ref{lem:Bernoulli-basic}, for $m\geq 1$ we have

\[
  2^{2m}|B_{2m}|
  \le 2^{2m}\cdot\frac{2(2m)!}{4^m\pi^{2m}}
  = 2\,(2m)!\,\pi^{-2m}.
\]
Since $B_s=0$ for odd $s\ge 3$, only even $s=2m$ with $1\le m\le\lfloor i/2\rfloor$
remain in the sum. Therefore
\[
  \sum_{s\ge 2} 2^{s}|B_s|\binom{i}{s}
  \le 2\sum_{m=1}^{\lfloor i/2\rfloor} (2m)!\,\binom{i}{2m}\,\pi^{-2m}\leq \sum_{m=1}^{\lfloor i/2\rfloor} \frac{2\,i!}{\pi^{2m}(i-2m)!}\;.
\]

Using $(i-2m)!\geq 1$ and bounding the sum by the infinite geometric sum, we conclude 
\[
  \sum_{s\ge 2} 2^{s}|B_s|\binom{i}{s}
  \leq 
  = \frac{2\,i!}{\pi^2-1}.
\]
Combining this with the $s=0,1$ terms gives
\[
  A_i
  \le 1 + i + \frac{2\,i!}{\pi^2-1}\leq i!\;.
\]
\end{proof}

\begin{proposition}[Right eigenvector $\ell_\infty$--bound]\label{prop:right-linf}
Let $n\ge 2$ and let $i\in E_n$ be even with $0\le i\le n-1$. Then for all
$1\le k\le n$,
\begin{equation}\label{eq:zeta-right-linf}
  |\zeta^{(i)}(k)|
  \le
  \begin{cases}
    1, & i=0,\\[4pt]
    4\,i!\,n^i, & i\ge 2.
  \end{cases}
\end{equation}
In particular, for every even $i\ge 2$ and $n\ge i+1$,
\[
  \|\zeta^{(i)}\|_{\infty} \le 4\,i!\,n^i.
\]
\end{proposition}
\begin{proof}
    Fix $i\in \mathcal{E}_n^{*}$. For $1\le k\le n$:
\[
  \zeta^{(i)}(k)
  = \sum_{t=0}^{i}
      2^{\,i-t} B_{i-t}\binom{i}{t}
      \frac{(n-t-1)!}{(n-i-1)!}\,(k-1)_t.
\]
We first bound each term in absolute value. For $1\le k\le n$ and
$0\le t\le i$,
\[
  |(k-1)_t|
  \le (k-1)^t \le n^t,
\]
and
\[
  \frac{(n-t-1)!}{(n-i-1)!}
  = (n-i-1)\cdots (n-t-1)
  \le n^{\,i-t}.
\]
Thus
\[
  \bigl|2^{\,i-t} B_{i-t}\binom{i}{t}
        \frac{(n-t-1)!}{(n-i-1)!}\,(k-1)_t\bigr|
  \le 2^{\,i-t} |B_{i-t}|\binom{i}{t}\,n^{\,i}.
\]
Re-indexing with $s=i-t$ gives
\[
  |\zeta^{(i)}(k)|
  \le n^i \sum_{s=0}^{i} 2^{s}|B_{s}|\binom{i}{s}
  =: n^i\,A_i\leq n^{i}\,i!
\]
where $A_i$ is as defined in Lemma~\ref{lem:A-i-bound}.
\end{proof}

\begin{proposition}[Left-eigenvector $\ell_{\infty}$-bound]
\label{prop:left-linf}
Let $n\ge 2$ and let $i\in \mathcal{E}_n$. Then
\begin{equation}\label{eq:left-linf}
\|\widetilde\zeta^{(i)}\|_\infty
:= \max_{1\le a\le n}|\widetilde\zeta^{(i)}(a)|
\le \frac{2^i}{i!(n-i)}.
\end{equation}
  
\end{proposition}

\begin{proof}
From \eqref{eq:zeta-tilde-explicit} we have, for each $a$,
\[
|\widetilde\zeta^{(i)}(a)|
\le \frac{1}{i!(n-i)}
     \Bigl(
        \binom{i-1}{a-1}
        + \binom{i-1}{a-1-(n-i)}
     \Bigr).
\]
Each binomial coefficient is at most
$\max_{0\le r\le i-1}\binom{i-1}{r}\le 2^{i-1}$, so
\[
|\widetilde\zeta^{(i)}(a)|
\le \frac{2\cdot 2^{i-1}}{i!(n-i)}
= \frac{2^i}{i!(n-i)},
\]
which gives \eqref{eq:left-linf}.
\end{proof}

\begin{proof}[Proof of Theorem~\ref{thm: Ciucu}]

Using the eigendecomposition of $M^k$ in~\eqref{eqn:M^kExapnsion} and the $\ell_{\infty}$ bound on the left and right eigenvectors of $M$ in Proposition~\ref{prop:right-linf} and Proposition~\ref{prop:left-linf}, we get 
\[\|M^k-\frac{1}{n}J\|_{\infty}\leq 4\sum_{i\in \mathcal{E}_n^{*}} 2^{-ik} \frac{(2n)^i}{n-i}\;.\]
Let $k\geq (1+\epsilon)\log n$ and assume that $n$ is sufficiently large. Then,  
\begin{align*}
 \|M^k-\frac{1}{n}J\|_{\infty}& \leq \frac{8}{n} \sum_{i\in \mathcal{E}_n^{*}, i\leq n/2} \left(\frac{2}{n^{\epsilon}}\right)^{i} + \left(\frac{2}{n^{\epsilon}}\right)^{n/2}\leq Cn^{-1-2\epsilon}\;.
\end{align*}

In particular, for any $j$ we have $|\max_{1\leq i\leq n} M^{(k)}(i, j)-\frac{1}{n}|\leq Cn^{-1-2\epsilon}$. Let $E_{n, k}$ be the optimal score after $k$ many single-shuffles without feedback. Then, 

\[|E_{n, k}-1| \leq \sum_{j=1}^{n}\left|\max_{1\leq i\leq n} M^{(k)}(i, j)-\frac{1}{n}\right|\leq Cn^{-2\epsilon}\;.\]

\end{proof}

\clearpage
\printbibliography

\end{document}